\documentclass[10pt,twocolumn]{article}
\usepackage[a4paper,top=15mm,bottom=17mm,left=14mm,right=14mm,columnsep=6mm]{geometry}
\usepackage[T1]{fontenc}
\usepackage[utf8]{inputenc}
\usepackage{lmodern}
\usepackage{microtype}
\usepackage{amsmath,amssymb,amsthm,mathtools}
\usepackage{bm}
\usepackage{booktabs,array,tabularx}
\usepackage{graphicx}
\usepackage{tikz}
\usetikzlibrary{arrows.meta,positioning,calc,fit}
\usepackage{enumitem}
\usepackage[hidelinks]{hyperref}
\usepackage{cite}
\usepackage{xcolor}
\usepackage{url}

\setlist[itemize]{leftmargin=*,nosep}
\newtheorem{theorem}{Theorem}[section]
\newtheorem{proposition}[theorem]{Proposition}
\newtheorem{lemma}[theorem]{Lemma}
\newtheorem{corollary}[theorem]{Corollary}
\newtheorem{definition}[theorem]{Definition}

\newcommand{\R}{\mathbb{R}}
\newcommand{\cF}{\mathcal{F}}
\newcommand{\cT}{\mathcal{T}}

\newcommand{\valid}{\mathrm{valid}}
\newcommand{\diag}{\operatorname{diag}}
\newcommand{\sgn}{\operatorname{sgn}}

\newcommand{\eps}{\varepsilon}

\title{\textbf{All-Minors Matrix-Tree Theory for Superport Networks:\\
Completed Quotient-Incidence Determinants and Conductance-Weighted Subdivision Extensions}}
\author{Tony Newton\\
Newton Astro Labs\thanks{Newton Astro Labs is the trading name under which Tony Newton conducts independent computational research in the United Kingdom; it is not a limited company.}\\
London, UK\\
\texttt{tony.newton79@gmail.com}}
\date{}

\begin{document}
\maketitle

\begin{abstract}
An electrical network can be summarized at its boundary by a response matrix: prescribed boundary voltages determine boundary currents. A superport network adds a constraint by grouping boundary terminals into superports, requiring the total current in each group to be zero and making voltage differences inside the groups the natural coordinates. Earlier work determined forest formulas for a single response entry and for the determinant of the whole response matrix. The missing case was an arbitrary subdeterminant, or \emph{minor}: one needs to know not only which spanning forests contribute, but also the sign carried by each forest. This paper supplies that sign rule.

After choosing one reference vertex in each superport, the response is \(L=(D^{T}K^{-1}D)^{-1}\), with \(K\) the grounded weighted Laplacian and \(D\) recording the selected voltage differences. Contracting the components of a physical spanning forest \(F\) produces a much smaller quotient port graph \(H_F\). Its reduced incidence matrix \(B_F=Q_FD\) has columns only of the forms \(0,\pm e_a,e_a-e_b\). Hence every square incidence minor is exactly \(0\) or \(\pm1\). For a \(k\)-set of response coordinates \(I\), append to \(B_F\) the selector rows \(E_I^T\) and define the completed quotient-incidence determinant
\(\widehat\chi_F(I)=\det\!\begin{psmallmatrix}B_F\\E_I^T\end{psmallmatrix}\).
For coordinate sets \(I,J\) of the same size, the arbitrary response minor is a weighted spanning-forest sum whose coefficient is simply \(\widehat\chi_F(I)\widehat\chi_F(J)\). Thus the Jacobi complementary-minor factors used in the derivation disappear from the final theorem. Direct block-triangular reduction gives \(\widehat\chi_F(I)\in\{0,\pm1\}\), nonzero exactly when the complementary quotient edges \(N\setminus I\) form a spanning tree of \(H_F\).

Two further consequences follow from the same structure. The unsigned principal forest numerators assemble into a multiaffine real-stable polynomial, giving Hadamard--Fischer log-submodularity and, after normalization, a strongly Rayleigh subset measure. The proof also extends from rooted superport coordinates to independent virtual dipole coordinates whose matrix is the incidence matrix of an oriented virtual forest, isolating graphic incidence as the essential mechanism. A comparison with Lam--Lo--Yuen's 2026 period-matrix formula explains why their related homological determinant factors can exceed one while the present graphic-incidence factors cannot. An enlarged exact audit checks 4,700 determinant/forest, virtual-dipole, and conductance-weighted subdivision identities with exact arithmetic and zero discrepancies. These computations are verification layers; the theorem statements rest on the algebraic proofs.
\end{abstract}


\section{Introduction}
Kirchhoff's matrix-tree theorem is a canonical bridge between linear algebra, graph combinatorics and electrical network theory \cite{kirchhoff,chaikenkleitman,bapat}. Its all-minors extensions replace the single spanning-tree polynomial by signed spanning-forest sums, with signs encoding how prescribed row and column sets are linked through forest components \cite{chaiken}. In an electrical network, Schur complementation transfers these ideas from the graph Laplacian to the boundary response matrix, where arbitrary minors admit signed grove expansions \cite{kenyonwilson}. Related determinant structures occur in resistance theory, inverse Laplacians and transfer-current formulas \cite{kleinrandic,bapatsiva,burtonpemantle}.

Superport networks add a boundary constraint. Boundary vertices are partitioned into groups, or superports, and the sum of incoming currents within each superport is constrained to vanish. Pylyavskyy, Shirokovskikh and Skopenkov established the basic theory and proved two matrix-tree endpoint results: a signed spanning-forest formula for individual response entries and a spanning-tree/valid-forest quotient for the full determinant \cite{superport}. Their Problem~8.1 asks for a spanning-forest formula for arbitrary response minors and identifies the sign structure as the central difficulty.

Throughout the main superport theorem, the physical network is a finite connected simple undirected graph with positive edge conductances, matching the convention of \cite{superport}. Parallel physical edges can be merged by adding their conductances, and self-loops are excluded from the formal physical graph; the forest quotient \(H_F\) is nevertheless allowed to be a multigraph and to contain loops. For a matrix \(A\), the notation \(A[S,T]\) means the submatrix with row set \(S\) and column set \(T\), both kept in their inherited order; \(A[:,T]\) means all rows and columns \(T\). These conventions are stated again with the formal setup in Section~2.

The response matrix itself is the linear rule connecting measurable boundary data. After one reference vertex is selected in each superport, let \(y\) collect the voltage differences from each nonreference boundary vertex to its superport reference, and let \(j\) collect the corresponding independent incoming currents. The superport response matrix \(L\) is defined by \(j=Ly\). Thus a response minor records the coupled response of selected current coordinates to selected voltage-difference coordinates.

The key observation here is that the superport response has the inverse-Gram form
\begin{equation}
L^{-1}=Z=D^{T}K^{-1}D.
\label{eq:inversegram}
\end{equation}
For a physical spanning forest \(F\), contraction of the forest components turns each column \(e_x-e_{r(x)}\) of \(D\) into an edge of a quotient port graph \(H_F\). The corresponding matrix \(B_F=Q_FD\) is not merely analogous to a graph-incidence matrix: it is exactly a reduced oriented incidence matrix. This fact resolves the sign problem because graphic incidence matrices are totally unimodular.

A closely related determinant-product structure appeared independently in the discrete period-matrix setting of Lam, Lo and Yuen \cite{lamloyuen}. Their minors are weighted sums over homological quasi-trees with coefficients given by products of two homological intersection determinants. The analogy is important and is now incorporated explicitly in the prior-art boundary. It does not invalidate the graphic \(0,\pm1\) result: the homological matrices of \cite{lamloyuen} need not be graphic incidence matrices and their determinants may have absolute value larger than one; only their maximal homological case is forced to \(0,\pm1\). The present superport coefficients are smaller for the structural reason proved in Lemma~\ref{lem:graphic} below.

The second strand of the paper concerns subdivision identities. Sun, Yang and Xu recently expressed a vertex-weighted Kirchhoff index through matchings of subdivisions and cycle-deletion corrections \cite{sunyangxu}. Their electrical edges have unit resistance. The same sign-reversing involution survives arbitrary positive conductances, producing conductance-weighted principal-minor, rooted-determinant and doubly weighted Kirchhoff formulas.

\subsection{Contributions and claim boundary}
The principal results are as follows. Theorem~\ref{thm:gram} expresses every minor of \(D^{T}K^{-1}D\) as a forest sum with quotient-incidence coefficients. Lemma~\ref{lem:graphic} strengthens the coefficient argument by deriving the exact column forms of \(B_F\) and giving an elementary total-unimodularity proof. Lemma~\ref{lem:completed} then packages the complementary-tree orientation directly into the completed determinant \(\widehat\chi_F(I)\). Theorem~\ref{thm:allminors} converts the Gram formula into an arbitrary superport response-minor formula with coefficient \(\widehat\chi_F(I)\widehat\chi_F(J)\), so no explicit Jacobi sign occurs in the theorem statement. Corollary~\ref{cor:principal} gives unsigned principal forest sums, and Theorem~\ref{thm:stable} upgrades the whole principal family to a real-stable forest polynomial. Theorem~\ref{thm:virtual} isolates a broader incidence-coordinate version for virtual dipole forests. Finally, Theorems~\ref{thm:subdivision}--\ref{thm:kirchhoffweighted} extend the subdivision-matching mechanism to arbitrary positive conductances.

The ingredients are not claimed as new individually. Jacobi's identity, Cauchy--Binet, all-minors matrix-tree formulas, graphic total unimodularity, grove formulas, stable determinantal polynomials, weighted subdivisions and the superport endpoint theorems are established prior art \cite{chaiken,kenyonwilson,schrijver,BBL,liyan,superport}. The claimed superport contribution is the quotient-incidence synthesis that supplies a closed forest sign rule for arbitrary response minors. The stability statement and virtual-dipole formulation are consequences of that synthesis; no claim of priority over unpublished general-boundary work is made.

A literature audit through 30 August 2026 located the period-matrix analogue of Lam--Lo--Yuen \cite{lamloyuen}, but no published theorem identified in the search that states an all-minors forest formula for the superport response matrix itself. The authors of \cite{superport} have also indicated to the author that a broader boundary-condition preprint is forthcoming. Because that manuscript was not publicly available at the time of this revision, the present paper does not attempt a priority comparison with it.

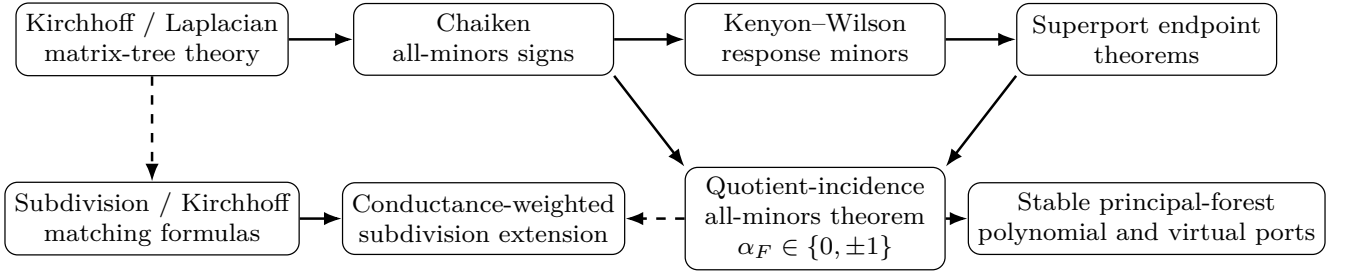
\begin{figure*}[t]
\centering
\resizebox{0.96\textwidth}{!}{%
\begin{tikzpicture}[every node/.style={font=\footnotesize,align=center}, box/.style={draw,rounded corners,minimum width=29mm,minimum height=8mm,inner sep=3pt}, arr/.style={-{Latex[length=2mm]},thick}]
\node[box] (kir) at (0,2.0) {Kirchhoff / Laplacian\\matrix-tree theory};
\node[box] (cha) at (3.7,2.0) {Chaiken\\all-minors signs};
\node[box] (kw)  at (7.4,2.0) {Kenyon--Wilson\\response minors};
\node[box] (sup) at (11.1,2.0) {Superport endpoint\\theorems};
\node[box] (sun) at (0,0) {Subdivision / Kirchhoff\\matching formulas};
\node[box] (cond) at (3.7,0) {Conductance-weighted\\subdivision extension};
\node[box] (main) at (7.4,0) {Quotient-incidence\\all-minors theorem\\\(\alpha_F\in\{0,\pm1\}\)};
\node[box] (stab) at (11.1,0) {Stable principal-forest\\polynomial and virtual ports};
\draw[arr] (kir)--(cha); \draw[arr] (cha)--(kw); \draw[arr] (kw)--(sup);
\draw[arr] (cha.south east)--(main.north west); \draw[arr] (sup.south west)--(main.north east); \draw[arr] (main)--(stab);
\draw[arr] (sun)--(cond);
\draw[arr,dashed] (kir)--(sun);
\draw[arr,dashed] (main.west)--(cond.east);
\end{tikzpicture}}
\caption{Logical position of the revised results. The upper lane is the all-minors response-matrix lineage; the lower lane separates the subdivision extension from the new quotient-incidence consequences. Solid arrows denote direct theorem dependence and dashed arrows denote a shared determinant/forest mechanism.}
\label{fig:logic}
\end{figure*}

\section{Weighted electrical networks and superports}
We follow the graph convention of \cite{superport}: unless stated otherwise, \(G=(V,E)\) is a finite connected \emph{simple} undirected graph, and each physical edge \(e\) carries a positive conductance \(c_e\in\R_{>0}\). Parallel physical edges may be merged by summing conductances. Self-loops are excluded from the formal setup because their Laplacian convention varies and, electrically, they carry no voltage drop. The quotient graph \(H_F\), by contrast, is deliberately a multigraph and may contain loops.

For any matrix \(A\) with ordered row and column index sets, \(A[S,T]\) denotes the submatrix with rows \(S\) and columns \(T\), in inherited ambient order; \(A[:,T]\) retains all rows. Complements are taken in the explicitly stated ambient ordered set. We use \(\det A[\varnothing,\varnothing]=1\).

The weighted Laplacian is
\begin{equation}
(L_c)_{uv}=\begin{cases}
\sum_{e\ni u}c_e,&u=v,\\
-c_{uv},&uv\in E,\\
0,&\text{otherwise.}
\end{cases}
\end{equation}
The weighted spanning-tree polynomial is
\begin{equation}
\tau_c(G)=\sum_{T\in\cT(G)}\prod_{e\in T}c_e,
\qquad w(F)=\prod_{e\in F}c_e.
\end{equation}

Let the boundary \(M\subseteq V\) be partitioned into nonempty superports
\begin{equation}
M=A_1\sqcup\cdots\sqcup A_p.
\end{equation}
The physical boundary condition imposes zero total incoming current in every \(A_a\). Choose a root \(r_a\in A_a\), put \(R=\{r_1,\ldots,r_p\}\), \(N=M\setminus R\), \(q=|N|\), and choose a distinguished ground root \(g\in R\). Write \(N=(x_1,\ldots,x_q)\) in the fixed coordinate order and \(r(x)\) for the root in the superport containing \(x\).

Delete the row and column of \(g\):
\begin{equation}
K=L_c[V\setminus\{g\},V\setminus\{g\}],\qquad \det K=\tau_c(G).
\end{equation}
In \(\R^{V\setminus\{g\}}\), set \(e_g=0\) and
\begin{equation}
d_x=e_x-e_{r(x)},\qquad D=[d_{x_1}\ \cdots\ d_{x_q}].
\end{equation}

\begin{definition}[superport response]
For each \(x\in N\), prescribe the voltage difference \(y_x=U_x-U_{r(x)}\). Let \(j_x\) denote the independent incoming current at \(x\); the current at each root is then fixed by the zero-total-current condition in its superport. The linear map \(y\mapsto j\) is the superport response, and its matrix in the chosen ordered root coordinates is denoted by \(L\):
\begin{equation}
j=Ly.
\end{equation}
\end{definition}

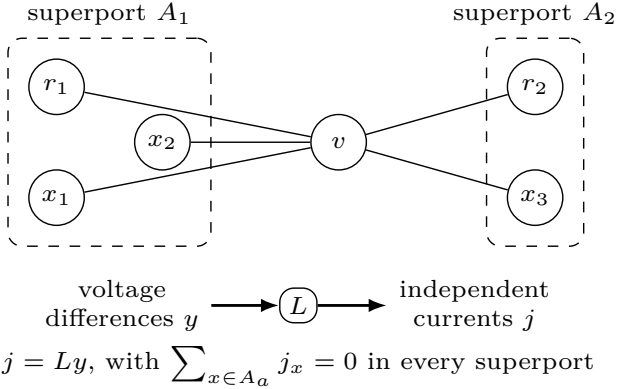
\begin{figure}[t]
\centering
\resizebox{0.96\columnwidth}{!}{%
\begin{tikzpicture}[every node/.style={font=\scriptsize}, v/.style={circle,draw,minimum size=5.5mm,inner sep=0pt}, box/.style={draw,rounded corners,inner sep=2.5pt}, arr/.style={-{Latex[length=1.8mm]},thick}]
\node[v] (r1) at (0,1.15) {$r_1$};
\node[v] (x1) at (0,0.05) {$x_1$};
\node[v] (x2) at (1.05,0.60) {$x_2$};
\node[v] (m) at (2.8,0.60) {$v$};
\node[v] (r2) at (4.75,1.15) {$r_2$};
\node[v] (x3) at (4.75,0.05) {$x_3$};
\draw (r1)--(m); \draw (x1)--(m); \draw (x2)--(m); \draw (m)--(r2); \draw (m)--(x3);
\node[draw,dashed,rounded corners,fit=(r1)(x1)(x2),inner sep=2.0mm,label=above:{\scriptsize superport $A_1$}] {};
\node[draw,dashed,rounded corners,fit=(r2)(x3),inner sep=2.0mm,label=above:{\scriptsize superport $A_2$}] {};
\node[align=center] (y) at (0.65,-1.02) {voltage\\differences $y$};
\node[box] (L) at (2.40,-1.02) {$L$};
\node[align=center] (j) at (4.15,-1.02) {independent\\currents $j$};
\draw[arr] (y)--(L); \draw[arr] (L)--(j);
\node[align=center] at (2.40,-1.62) {$j=Ly$, with $\sum_{x\in A_a}j_x=0$ in every superport};
\end{tikzpicture}}
\caption{What the superport response matrix records. A root $r_a$ is chosen in each boundary group, voltages are represented by differences from that root, and the independent currents are constrained so that each superport has zero net current. The matrix $L$ maps the voltage-difference vector $y$ to the current vector $j$.}
\label{fig:responseconcept}
\end{figure}

\begin{proposition}[inverse-Gram representation]\label{prop:gram}
The superport response satisfies
\begin{equation}
L=(D^{T}K^{-1}D)^{-1}.
\end{equation}
\end{proposition}
\begin{proof}
The grounded node-injection vector is \(Dj\). Kirchhoff's law gives \(Ku=Dj\), hence \(u=K^{-1}Dj\). The coordinate voltage differences are \(y=D^{T}u=D^{T}K^{-1}Dj\). Since \(j=Ly\), equation \eqref{eq:inversegram} follows.
\end{proof}

The matrix \(Z:=D^{T}K^{-1}D\) is symmetric positive definite. It may be viewed as a Green-function Gram matrix of virtual port-difference edges, in the spirit of transfer-current constructions \cite{burtonpemantle}.

\section{Quotient-incidence calculus for spanning forests}
\begin{definition}[quotient port graph]
Let \(F\) be a spanning forest of \(G\) with \(s\ge1\) connected components. Let \(C_0(F)\) be the component containing \(g\), and order the others as \(C_1(F),\ldots,C_{s-1}(F)\). The quotient port multigraph \(H_F\) has these forest components as vertices. Each coordinate \(x\in N\) gives a directed port edge from the component containing \(r(x)\) to the component containing \(x\); if both endpoints lie in one forest component, the quotient edge is a loop.
\end{definition}

\begin{figure*}[t]
\centering
\resizebox{0.90\textwidth}{!}{%
\begin{tikzpicture}[every node/.style={font=\footnotesize}, pv/.style={circle,draw,minimum size=5mm,inner sep=0pt}, cv/.style={circle,draw,minimum size=9mm,inner sep=0pt}, arr/.style={-{Latex[length=2mm]},thick}]
\node[pv] (a0) at (0,1.2) {};
\node[pv] (a1) at (0.9,1.2) {};
\node[pv] (b0) at (2.0,2.0) {};
\node[pv] (b1) at (2.9,2.0) {};
\node[pv] (c0) at (2.0,0.3) {};
\node[pv] (c1) at (2.9,0.3) {};
\draw[very thick] (a0)--(a1); \draw[very thick] (b0)--(b1); \draw[very thick] (c0)--(c1);
\draw[-{Latex[length=1.7mm]},dashed] (a1) -- node[above left] {$x_1$} (b0);
\draw[-{Latex[length=1.7mm]},dashed] (a1) -- node[below left] {$x_3$} (c0);
\draw[-{Latex[length=1.7mm]},dashed] (b1) to[bend left=18] node[right] {$x_2$} (c1);
\node[align=center] at (1.45,-0.5) {physical spanning forest $F$\\three connected components};
\draw[arr] (3.7,1.15)--node[above]{contract each forest component}(6.0,1.15);
\node[cv] (C0) at (7.0,1.15) {$C_0$};
\node[cv] (C1) at (9.0,2.0) {$C_1$};
\node[cv] (C2) at (9.0,0.3) {$C_2$};
\draw[arr] (C0)--node[above left]{$x_1$}(C1);
\draw[arr] (C0)--node[below left]{$x_3$}(C2);
\draw[arr] (C1) to[bend left=18] node[right]{$x_2$}(C2);
\node[align=center] at (8.0,-0.5) {quotient port graph $H_F$\\its incidence minors carry the signs};
\end{tikzpicture}}
\caption{The key reduction behind the theorem. The physical forest may be large, but after each connected forest component is contracted to one vertex, the selected superport voltage-difference coordinates become directed edges of the quotient graph $H_F$. The sign calculation is then performed on this smaller graph.}
\label{fig:contraction}
\end{figure*}
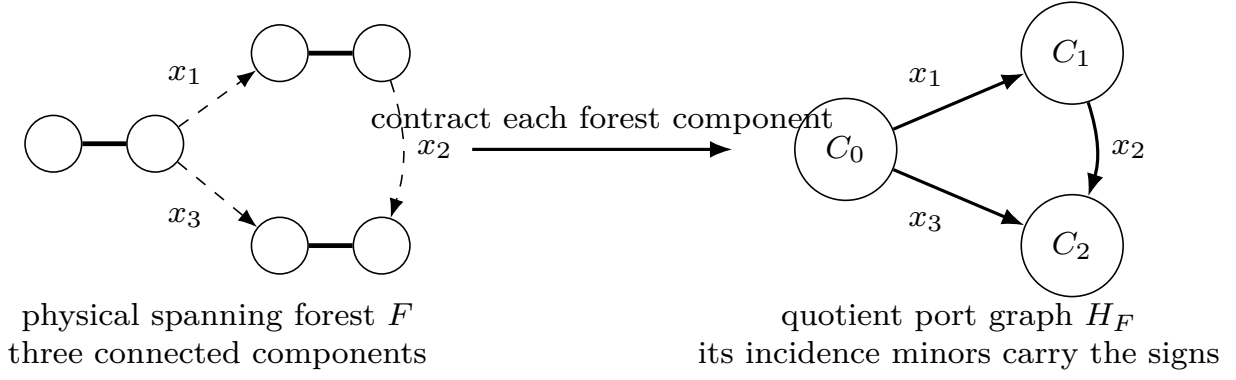

Let \(Q_F\) be the \((s-1)\times(|V|-1)\) non-ground component-indicator matrix,
\begin{equation}
(Q_F)_{a,v}=\begin{cases}1,&v\in C_a(F),\\0,&\text{otherwise,}\end{cases}
\qquad B_F=Q_FD.
\end{equation}
For \(S\subseteq N\), \(|S|=s-1\), define
\begin{equation}
\chi_F(S)=\det B_F[:,S],\qquad \chi_F(\varnothing)=1\ (s=1).
\end{equation}

\begin{figure}[t]
\centering
\begin{tikzpicture}[scale=.82, every node/.style={font=\footnotesize}]
\node[circle,draw,minimum size=7mm] (c0) at (0,0) {$C_0$};
\node[circle,draw,minimum size=7mm] (c1) at (2.2,1) {$C_1$};
\node[circle,draw,minimum size=7mm] (c2) at (2.2,-1) {$C_2$};
\draw[-{Latex[length=2mm]}] (c0) -- node[above] {$x_1$} (c1);
\draw[-{Latex[length=2mm]}] (c0) -- node[below] {$x_3$} (c2);
\draw[-{Latex[length=2mm]}] (c1) to[bend left=18] node[right] {$x_2$} (c2);
\draw[-{Latex[length=2mm]}] (c2) to[bend left=18] node[left] {$x_4$} (c1);
\node[below=4mm of c0] {ground component};
\end{tikzpicture}
\caption{A quotient port graph \(H_F\). Physical forest components become vertices; port-difference coordinates become directed quotient edges.}
\label{fig:quotient}
\end{figure}
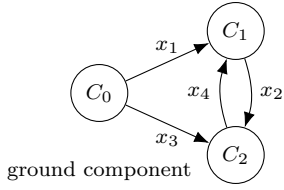

\begin{lemma}[graphic unimodularity]\label{lem:graphic}
For every spanning forest \(F\), every square minor of \(B_F\) belongs to \(\{0,+1,-1\}\). In particular, if \(|S|=s-1\),
\begin{equation}
\chi_F(S)\in\{0,+1,-1\},
\end{equation}
and \(\chi_F(S)\neq0\) if and only if the quotient edges in \(S\) form a spanning tree of \(H_F\).
\end{lemma}
\begin{proof}
Define \(\kappa_F(v)=0\) when \(v\in C_0(F)\) and \(\kappa_F(v)=e_a\) when \(v\in C_a(F)\), \(a\ge1\). For a coordinate \(x\),
\begin{equation}
B_F[:,x]=Q_F(e_x-e_{r(x)})
=\kappa_F(x)-\kappa_F(r(x)).
\label{eq:columnform}
\end{equation}
Hence every column is exactly one of
\begin{equation}
0,\quad +e_a,\quad -e_a,\quad e_a-e_b.
\end{equation}
Thus \(B_F\) is the reduced oriented incidence matrix of \(H_F\).

For completeness, total unimodularity follows elementarily. Consider any square submatrix \(C\). If some column has at most one nonzero entry, expand \(\det C\) along that column and induct on its order. If every column has two nonzero entries, each such column contains one \(+1\) and one \(-1\); the selected rows then sum to zero, so \(\det C=0\). Therefore every square minor is \(0\) or \(\pm1\). For a maximal \((s-1)\times(s-1)\) incidence minor, nonzero determinant is equivalent to choosing a basis of the graphic matroid, i.e. a spanning tree of \(H_F\) \cite{schrijver,oxley}.
\end{proof}

This proof also clarifies the contrast with \cite{lamloyuen}: their determinant factors are homological intersection minors, not reduced graphic-incidence minors. Such homological minors can encode nonprimitive sublattices and need not be unimodular in nonmaximal degree.

\subsection{Completed quotient-incidence determinants}
For \(I\subseteq N\), \(|I|=k\), let \(E_I\) be the \(q\times k\) selector matrix whose columns are the standard basis vectors indexed by \(I\), in inherited order. If \(F\in\cF_{q-k+1}\), then \(B_F\) has \(q-k\) rows and the stacked matrix below is square. Define
\begin{equation}
\boxed{\widehat\chi_F(I):=
\det\begin{pmatrix}B_F\\[1mm]E_I^T\end{pmatrix}.}
\label{eq:completedchi}
\end{equation}
The selector rows complete the reduced incidence matrix to a full \(q\times q\) determinant. They are the device that absorbs the positional parity of complementary minors into the matrix ordering itself.

\begin{lemma}[completed-incidence tree criterion]\label{lem:completed}
For \(F\in\cF_{q-k+1}\) and \(|I|=k\),
\begin{equation}
\widehat\chi_F(I)\in\{0,+1,-1\},
\end{equation}
and
\begin{equation}
\widehat\chi_F(I)\neq0
\quad\Longleftrightarrow\quad
N\setminus I\text{ is a spanning tree of }H_F.
\label{eq:completedtree}
\end{equation}
Moreover, if \(\eps_N(I)\) denotes the proof-level Jacobi orientation sign defined below, then
\begin{equation}
\widehat\chi_F(I)=(-1)^{k(q+1)}\eps_N(I)\chi_F(N\setminus I).
\label{eq:completedbridge}
\end{equation}
\end{lemma}
\begin{proof}
Permute the \(q\) columns into the inherited order \((N\setminus I,I)\). The completed matrix becomes
\begin{equation}
\begin{pmatrix}
B_F[:,N\setminus I] & B_F[:,I]\\
0&I_k
\end{pmatrix}.
\label{eq:completedblock}
\end{equation}
It is block upper triangular, so the absolute value of \(\widehat\chi_F(I)\) equals \(|\det B_F[:,N\setminus I]|\). Lemma~\ref{lem:graphic} therefore proves both the \(0,\pm1\) range and the spanning-tree criterion directly, without Jacobi's identity. Tracking the sign of the column permutation gives \eqref{eq:completedbridge}; the exponents \(k(q-k)\) and \(k(q+1)\) are congruent modulo two because \(k(k+1)\) is even.
\end{proof}

\section{Forest formula for inverse-Gram minors}
For the derivation only, let \(U\) be an ordered finite set and \(S=\{s_1<\cdots<s_r\}\subseteq U\), and define the Jacobi orientation sign
\begin{equation}
\eps_U(S)=(-1)^{\sum_{a=1}^{r}\operatorname{pos}_U(s_a)-r(r+1)/2}.
\label{eq:epsilon}
\end{equation}
It is the sign of the permutation that lists \(S\) first and \(U\setminus S\) second, both in inherited order. For equal-cardinality \(S,T\),
\begin{equation}
\eps_U(S)\eps_U(T)=(-1)^{\sigma_U(S)+\sigma_U(T)},
\end{equation}
where \(\sigma_U\) is the sum of one-based positions. Thus the earlier \(\sigma\)-notation is simply Jacobi parity, not an additional superport invariant. Lemma~\ref{lem:completed} absorbs this proof-level parity into \(\widehat\chi_F\), so it does not appear in the final response-minor theorem.

\begin{lemma}[grounded inverse-minor forest identity]\label{lem:inverseforest}
Let \(S,T\subseteq V\setminus\{g\}\) with \(|S|=|T|=r\). Then
\begin{align}
\det K^{-1}[S,T]
&=\frac{1}{\tau_c(G)}
  \sum_{F\in\cF_{r+1}}
  \det Q_F[:,S]\notag\\
&\hspace{19mm}{}\times\det Q_F[:,T]w(F).
\label{eq:groundedinverse}
\end{align}
where \(\cF_{r+1}\) denotes spanning forests with exactly \(r+1\) components.
\end{lemma}
\begin{proof}
Jacobi's complementary-minor identity gives
\begin{equation}
\det K^{-1}[S,T]
=\eps_{V\setminus\{g\}}(S)\eps_{V\setminus\{g\}}(T)
 \frac{\det K[T^c,S^c]}{\det K}.
\end{equation}
Chaiken's all-minors matrix-tree theorem expands the signed complementary minor as the forest sum in \eqref{eq:groundedinverse}; \(\det K=\tau_c(G)\) by the weighted matrix-tree theorem \cite{chaiken}.
\end{proof}

\begin{theorem}[inverse-Gram forest-minor theorem]\label{thm:gram}
Let \(A,B\subseteq N\) with \(|A|=|B|=r\). Then
\begin{align}
\det Z[A,B]
&=\frac{1}{\tau_c(G)}
  \sum_{F\in\cF_{r+1}}
  \chi_F(A)\chi_F(B)w(F).
\label{eq:gramforest}
\end{align}
\end{theorem}
\begin{proof}
Write \(D_A=D[:,A]\) and \(D_B=D[:,B]\). Double Cauchy--Binet gives
\begin{align}
\det(D_A^TK^{-1}D_B)
={}&\sum_{S,T}\det D[S,A]\det K^{-1}[S,T]\notag\\
&\qquad\times\det D[T,B].
\end{align}
Insert Lemma~\ref{lem:inverseforest}, interchange the finite sums, and apply Cauchy--Binet once more:
\begin{equation}
\sum_S\det Q_F[:,S]\det D[S,A]
=\det(Q_FD_A)=\chi_F(A),
\end{equation}
with the analogous identity for \(B\).
\end{proof}

\section{The all-minors superport matrix-tree theorem}
A spanning forest is \emph{valid} if identifying all boundary vertices within each superport turns it into a spanning tree of the quotient network \cite{superport}. Define
\begin{equation}
\Phi_{\valid}=\sum_{F\ \valid}w(F).
\end{equation}

\begin{proposition}[valid-forest denominator]\label{prop:denom}
\begin{equation}
\det Z=\frac{\Phi_{\valid}}{\tau_c(G)},
\qquad
\det L=\frac{\tau_c(G)}{\Phi_{\valid}}.
\end{equation}
\end{proposition}
\begin{proof}
Set \(A=B=N\) in Theorem~\ref{thm:gram}. The relevant forests have \(q+1\) components and
\begin{equation}
\det Z=\frac1{\tau_c(G)}\sum_{F\in\cF_{q+1}}\chi_F(N)^2w(F).
\end{equation}
By Lemma~\ref{lem:graphic}, \(\chi_F(N)^2\) is one exactly when the full quotient port-edge set is a spanning tree of \(H_F\), which is equivalent to validity.
\end{proof}

\begin{theorem}[all-minors matrix-tree theorem for superport networks]\label{thm:allminors}
Let \(I,J\subseteq N\) with \(|I|=|J|=k\), and put \(r=q-k\). Then
\begin{equation}
\boxed{
\det L[I,J]
=\frac{1}{\Phi_{\valid}}
\sum_{F\in\cF_{r+1}}
\widehat\chi_F(I)\widehat\chi_F(J)w(F).}
\label{eq:mainall}
\end{equation}
For every forest,
\begin{equation}
\boxed{\alpha_F(I,J):=\widehat\chi_F(I)\widehat\chi_F(J)\in\{0,+1,-1\}.}
\label{eq:alphahat}
\end{equation}
The coefficient is nonzero exactly when both complementary port-edge sets \(N\setminus I\) and \(N\setminus J\) are spanning trees of \(H_F\).
\end{theorem}
\begin{proof}
Since \(L=Z^{-1}\), Jacobi's complementary-minor identity expresses \(\det L[I,J]\) as the complementary minor \(\det Z[N\setminus J,N\setminus I]\) divided by \(\det Z\), with the standard orientation factors \(\eps_N(I)\eps_N(J)\). Apply Theorem~\ref{thm:gram} to the numerator and Proposition~\ref{prop:denom} to the denominator. Lemma~\ref{lem:completed}, equation~\eqref{eq:completedbridge}, replaces each product \(\eps_N(\cdot)\chi_F(N\setminus\cdot)\) by the corresponding completed determinant up to the common factor \((-1)^{k(q+1)}\). The two common factors multiply to one. This gives \eqref{eq:mainall}. The coefficient range and the two-tree criterion follow directly from \eqref{eq:completedtree}.
\end{proof}

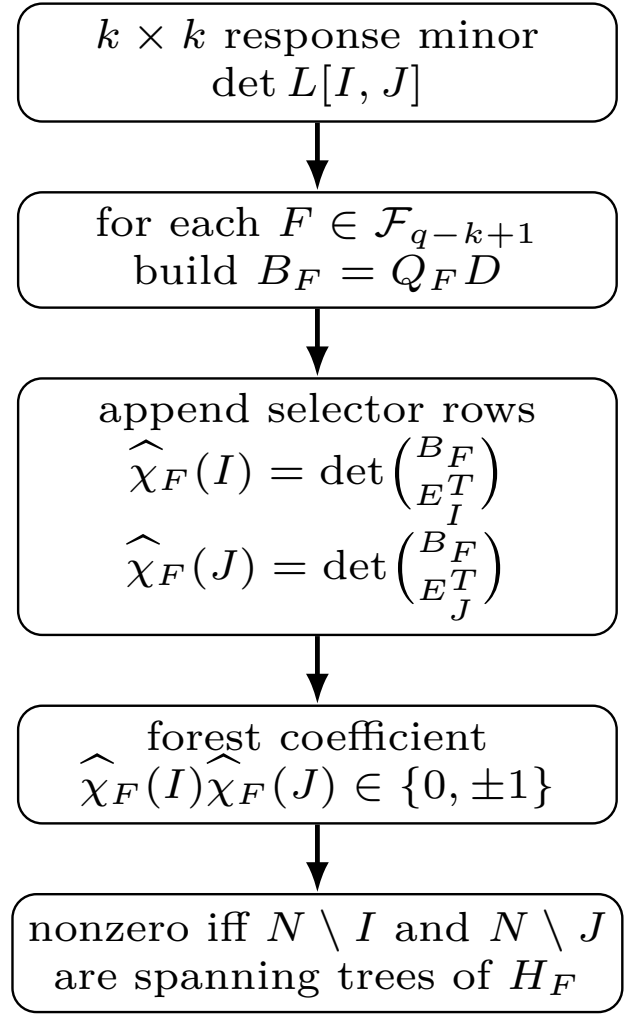
\begin{figure}[t]
\centering
\resizebox{0.92\columnwidth}{!}{%
\begin{tikzpicture}[every node/.style={font=\scriptsize,align=center}, box/.style={draw,rounded corners,inner sep=2.5pt,minimum width=35mm}, arr/.style={-{Latex[length=1.8mm]},thick}, node distance=4mm]
\node[box] (minor) {$k\times k$ response minor\\$\det L[I,J]$};
\node[box,below=of minor] (forest) {for each $F\in\mathcal F_{q-k+1}$\\build $B_F=Q_FD$};
\node[box,below=of forest] (complete) {append selector rows\\$\widehat\chi_F(I)=\det\!\binom{B_F}{E_I^T}$\\$\widehat\chi_F(J)=\det\!\binom{B_F}{E_J^T}$};
\node[box,below=of complete] (coeff) {forest coefficient\\$\widehat\chi_F(I)\widehat\chi_F(J)\in\{0,\pm1\}$};
\node[box,below=of coeff] (tree) {nonzero iff $N\setminus I$ and $N\setminus J$\\are spanning trees of $H_F$};
\draw[arr] (minor)--(forest); \draw[arr](forest)--(complete); \draw[arr](complete)--(coeff); \draw[arr](coeff)--(tree);
\end{tikzpicture}}
\caption{How to read Theorem~\ref{thm:allminors}. The selector-row completion absorbs complementary-minor parity into an ordinary determinant. The final coefficient is the product of two completed quotient-incidence determinants; no explicit Jacobi sign remains in the theorem statement.}
\label{fig:minortest}
\end{figure}

A contributing \(F\in\cF_{q-k+1}\) gives a connected quotient \(H_F\) with \(q-k+1\) vertices and \(q\) port edges. Therefore
\begin{equation}
\beta_1(H_F)=q-(q-k+1)+1=k.
\end{equation}
Thus the response-minor order equals the cycle rank of the supporting quotient port graph.

\begin{corollary}[unsigned principal minors]\label{cor:principal}
For \(I\subseteq N\), \(|I|=k\), define
\begin{equation}
\Psi_I=
\sum_{\substack{F\in\cF_{q-k+1}\\N\setminus I\text{ spans a tree of }H_F}}w(F),
\qquad \Psi_{\varnothing}:=\Phi_{\valid}.
\end{equation}
Then
\begin{equation}
\boxed{\det L[I,I]=\frac{\Psi_I}{\Phi_{\valid}}.}
\label{eq:principal}
\end{equation}
\end{corollary}
\begin{proof}
For \(I=J\), equation~\eqref{eq:mainall} contains \(\widehat\chi_F(I)^2\in\{0,1\}\), which is one exactly on the stated complementary spanning-tree condition.
\end{proof}

\subsection{Principal-forest stability and inequalities}
The unsigned forest family carries more structure than termwise positivity.

\begin{theorem}[real-stable principal-forest polynomial]\label{thm:stable}
Let \(X(z)=\diag(z_1,\ldots,z_q)\) and define
\begin{equation}
\mathcal P_F(z_1,\ldots,z_q)
:=\sum_{I\subseteq N}\Psi_I\prod_{i\in I}z_i.
\end{equation}
Then
\begin{equation}
\boxed{\mathcal P_F(z)=\Phi_{\valid}\det(I_q+X(z)L),}
\label{eq:stablepoly}
\end{equation}
and \(\mathcal P_F\) is multiaffine, coefficientwise nonnegative and real stable.
\end{theorem}
\begin{proof}
The principal-minor expansion gives
\begin{equation}
\det(I_q+XL)=\sum_{I\subseteq N}\det L[I,I]\prod_{i\in I}z_i.
\end{equation}
Insert Corollary~\ref{cor:principal}. For stability, use \(\det(I+XL)=\det(I+L^{1/2}XL^{1/2})\). If every \(\operatorname{Im}z_i>0\), the imaginary part of \(I+L^{1/2}XL^{1/2}\) is
\(\sum_i(\operatorname{Im}z_i)v_iv_i^T\), where the \(v_i\) are the columns of \(L^{1/2}\). Since \(L\) is positive definite, these vectors span \(\R^q\); the imaginary part is therefore positive definite and the determinant cannot vanish. This is the standard determinantal stability mechanism \cite{BBL}.
\end{proof}

\begin{corollary}[Hadamard--Fischer forest inequality]\label{cor:fischer}
For all \(I,J\subseteq N\),
\begin{equation}
\boxed{\Psi_I\Psi_J\ge \Psi_{I\cup J}\Psi_{I\cap J}.}
\end{equation}
After normalization by \(\mathcal P_F(1,\ldots,1)\), the coefficients define a strongly Rayleigh probability measure on subsets of \(N\).
\end{corollary}
\begin{proof}
The matrix inequality follows from Hadamard--Fischer log-submodularity of principal minors of a positive definite matrix \cite{hornjohnson}; the common denominator \(\Phi_{\valid}\) cancels using \eqref{eq:principal}. The strongly Rayleigh conclusion follows from real stability and nonnegative coefficients \cite{BBL}.
\end{proof}

\subsection{Endpoint reductions}
If \(I=J=N\), then \(k=q\), the numerator forests are spanning trees and \(\det L=\tau_c(G)/\Phi_{\valid}\), recovering \cite{superport}. If \(k=1\), contributing quotient graphs are unicyclic and the two complementary spanning-tree orientations reproduce the signed entry structure of the published endpoint theorem. If all boundary vertices lie in a single superport, the root-difference coordinates reduce to the standard grounded response setting and the formula specializes to ordinary response/grove all-minors theory \cite{kenyonwilson}.

\section{Coordinate covariance and an incidence-coordinate extension}
\subsection{Root covariance and virtual dipole forests}
A root choice is a coordinate choice, not new network physics. If another root system gives \(D'=DU\) with a block-unimodular \(U\in GL_q(\mathbb Z)\), then
\begin{equation}
Z'=U^TZU,\qquad L'=U^{-1}LU^{-T},
\end{equation}
and \(B_F'=B_FU\). The forest formulas transform by the same Cauchy--Binet law as the matrix minors. This includes replacing star coordinates within a superport by any tree basis of pairwise voltage differences.

The proof actually requires still less: a physical superport partition is not needed for Theorem~\ref{thm:gram}. It is enough that the coordinate matrix is itself a reduced incidence matrix of independent virtual dipoles.

\begin{theorem}[virtual-dipole forest extension]\label{thm:virtual}
Let \(P\) be an oriented forest of \(q\) virtual edges on vertices of \(G\) (the virtual edges need not be physical edges and may involve interior vertices). Let \(D_P\) be its reduced incidence matrix after grounding \(g\), and put
\begin{equation}
Z_P=D_P^TK^{-1}D_P,\qquad \mathcal L_P=Z_P^{-1}.
\end{equation}
For each physical spanning forest \(F\), contract its components in \(P\) to obtain a quotient virtual-port multigraph \(H_F^P\) with reduced incidence \(B_F^P=Q_FD_P\). Define \(\chi_F^P(S)=\det B_F^P[:,S]\). Then every inverse-Gram minor satisfies
\begin{equation}
\det Z_P[A,B]
=\frac{1}{\tau_c(G)}\sum_{F\in\cF_{r+1}}
\chi_F^P(A)\chi_F^P(B)w(F).
\end{equation}
If
\begin{equation}
\Phi_P:=\sum_{\substack{F\in\cF_{q+1}\\E(P)\text{ spans a tree of }H_F^P}}w(F),
\end{equation}
then for \(|I|=|J|=k\), \(r=q-k\), let \(E_I\) denote the selector matrix in the inherited virtual-edge order and define
\begin{equation}
\widehat\chi_F^P(I):=\det\begin{pmatrix}B_F^P\\E_I^T\end{pmatrix}.
\end{equation}
Then
\begin{align}
\det\mathcal L_P[I,J]
&=\frac{1}{\Phi_P}
  \sum_{F\in\cF_{r+1}}
  \widehat\chi_F^P(I)\widehat\chi_F^P(J)w(F).
\label{eq:virtualresponse}
\end{align}
All coefficients remain \(0\) or \(\pm1\), and a completed determinant is nonzero exactly when its complementary virtual-edge set is a spanning tree of \(H_F^P\).
\end{theorem}
\begin{proof}
Theorem~\ref{thm:gram} used only Cauchy--Binet and Lemma~\ref{lem:inverseforest}; replacing \(D\) by \(D_P\) is therefore immediate. Since \(P\) is a forest, its incidence columns are independent, so \(Z_P\) is positive definite. Equation \eqref{eq:columnform} continues to hold for every virtual edge, hence \(B_F^P\) is a reduced incidence matrix and Lemma~\ref{lem:graphic} applies. The selector-row completion used in Lemma~\ref{lem:completed} applies verbatim to \(B_F^P\); the same Jacobi-to-completion cancellation as in Theorem~\ref{thm:allminors} gives the sigma-free response formula.
\end{proof}

When the virtual edges are the root stars inside boundary superports, Theorem~\ref{thm:virtual} is exactly the superport theorem. The extension is included to expose the algebraic hypothesis and to provide a testable bridge to more general boundary-coordinate formalisms. Because broader boundary-condition work has been announced by the authors of \cite{superport}, no novelty claim is made here for the most general possible formulation.

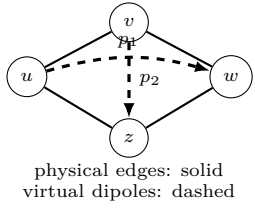
\begin{figure}[t]
\centering
\begin{tikzpicture}[every node/.style={font=\scriptsize}, scale=.9]
\node[draw,circle] (a) at (0,0) {$u$};
\node[draw,circle] (b) at (1.5,.8) {$v$};
\node[draw,circle] (c) at (3,0) {$w$};
\node[draw,circle] (d) at (1.5,-.9) {$z$};
\draw[thick] (a)--(b); \draw[thick] (b)--(c); \draw[thick] (c)--(d); \draw[thick] (d)--(a);
\draw[-{Latex[length=2mm]},very thick,dashed] (a) to[bend left=16] node[above] {$p_1$} (c);
\draw[-{Latex[length=2mm]},very thick,dashed] (b) -- node[right] {$p_2$} (d);
\node[align=center] at (1.5,-1.55) {physical edges: solid\\virtual dipoles: dashed};
\end{tikzpicture}
\caption{The incidence-coordinate extension permits independent virtual dipoles not restricted to physical edges. Their contraction through a physical forest again produces a graphic quotient incidence matrix.}
\label{fig:virtual}
\end{figure}

\section{Conductance-weighted subdivision identities}
Sun, Yang and Xu \cite{sunyangxu} consider vertex weights with unit electrical edge resistance. Their determinant involution extends to arbitrary positive conductances. Let \(S_c(G)\) be the subdivision graph in which each incidence edge \(ve^*\), \(e\ni v\), has weight \(c_e\). For a weighted graph \(H\), write \(m(H,j)\) for the total weight of its \(j\)-edge matchings. Let \(\mathcal C(G)\) be the set of nonempty vertex-disjoint unions of cycles; for \(C\in\mathcal C(G)\), let \(p(C)\) be its number of cycle components and \(c(C)=\prod_{e\in E(C)}c_e\).

\subsection{Weighted cancellation and identities}
For a physical edge \(e=uv\), a transposition \((uv)\) in the Leibniz expansion contributes \(-c_e^2\) relative to the remainder of the term. If \(u,v\) are fixed points and both diagonal sums choose \(e\), the paired configuration contributes \(+c_e^2\). Hence
\begin{equation}
-c_e^2+c_e^2=0.
\end{equation}
An oriented cycle contributes minus the product of its conductances; the two orientations of each undirected cycle therefore give \(-2c(C)\).

\begin{theorem}[conductance-weighted principal-minor subdivision identity]\label{thm:subdivision}
For every \(X\subseteq V\), put \(R_X:=V\setminus X\),
\(S_X:=S_c(G)-R_X\), and \(S_{X,C}:=S_c(G)-S(C)-R_X\). Then
\begin{align}
\det L_c[X,X]
={}&m(S_X,|X|)\notag\\
&+\sum_{\substack{C\in\mathcal C(G)\\V(C)\subseteq X}}
 (-2)^{p(C)}c(C)\notag\\
&\qquad\times m(S_{X,C},|X|-|C|).
\label{eq:subdivision}
\end{align}
\end{theorem}
\begin{proof}
Expand every diagonal entry as \(\sum_{e\ni v}c_e\). Pair each transposition with the corresponding two-fixed-point configuration choosing the same edge; these cancel exactly. The remaining nontrivial permutation cycles form a vertex-disjoint 2-regular subgraph \(C\), producing \((-2)^{p(C)}c(C)\). The remaining fixed vertices choose distinct incident edges, in weight-preserving bijection with a subdivision matching saturating those original vertices.
\end{proof}

For \(X=V\setminus\{u,v\}\), the weighted resistance identity
\begin{equation}
r_c(u,v)=\frac{\det L_c[X,X]}{\tau_c(G)}
\end{equation}
turns Theorem~\ref{thm:subdivision} into a conductance-weighted matching formula for pair resistance.

Let \(x_v>0\) be vertex weights and
\begin{equation}
K_c(G;x)=\sum_{\{u,v\}\subseteq V}x_ux_vr_c(u,v).
\end{equation}
Construct \(S_{c/x}(G)\) by assigning \(ve^*\) the weight \(c_e/x_v\), and write
\(S_{c/x}^{C}(G):=S_{c/x}(G)-S(C)\).

\begin{theorem}[doubly weighted Kirchhoff--subdivision formula]\label{thm:kirchhoffweighted}
For a connected conductance-weighted graph on \(n\) vertices,
\begin{align}
K_c(G;x)
={}&\frac1{\tau_c(G)}\Bigg\{
 m\bigl(S_{c/x}(G),n-2\bigr)
 \prod_{v\in V}x_v\notag\\
&+\sum_{C\in\mathcal C(G)}(-2)^{p(C)}c(C)\notag\\
&\qquad\times m\bigl(S_{c/x}^{C}(G),n-2-|C|\bigr)\notag\\
&\qquad\times\prod_{v\notin V(C)}x_v\Bigg\}.
\label{eq:doublyweighted}
\end{align}
\end{theorem}
\begin{proof}
Apply Theorem~\ref{thm:subdivision} to \(X=V\setminus\{u,v\}\), divide by \(\tau_c(G)\), multiply by \(x_ux_v\), and sum over unordered pairs. Equivalently, extract square-free coefficients in the vertex-weight variables. The inverse vertex factors on saturated original vertices cancel, while \(c(C)\) supplies the cycle conductances.
\end{proof}

If every \(c_e=1\), this reduces to the 2026 formula of Sun--Yang--Xu \cite{sunyangxu}; with all \(x_v=1\), it reduces to the ordinary subdivision formula of Que--Chen \cite{quechen}. Weighted subdivision graphs already occur in Laplacian matching-polynomial theory \cite{liyan}; the present statement is the cycle-corrected physical-conductance specialization.

A root-augmented master determinant is obtained by attaching to every original vertex \(v\) a private leaf \(\rho_v\) of weight \(z_v\). If \(Z_0=\diag(z_v)\), the same cancellation gives
\begin{align}
\det(L_c+Z_0)
={}&m\bigl(\widehat S_{c,z}(G),n\bigr)\notag\\
&+\sum_{C\in\mathcal C(G)}(-2)^{p(C)}c(C)\notag\\
&\qquad\times
 m\bigl(\widehat S_{c,z}(G)-S(C),n-|C|\bigr),
\label{eq:rootaugmented}
\end{align}
whose coefficients recover all principal Laplacian minors.

\section{Synthesis, verification, and prior-art boundary}
The two main strands should not be conflated. The superport theorem is a forest theorem for a constrained response operator; the subdivision theorem is a matching expansion of physical Laplacian determinants. Their common architecture is
\begin{equation}
\begin{gathered}
\text{determinant}\longrightarrow\text{all-minors/cancellation}\\
\longrightarrow\text{combinatorial basis objects}.
\end{gathered}
\end{equation}
In the superport route, total unimodularity performs the final sign reduction. In the subdivision route, the surviving basis object is a 2-regular cycle skeleton plus a matching of the residual subdivision.

\subsection{Exact computational verification}
The proofs are algebraic. Exact computation is used only as an independent error-detection layer for complement order, proof-level Jacobi parity, root choices, incidence orientation, selector-row completion and cycle weights. All conductances and vertex weights in the campaigns were rational; all determinants were evaluated exactly.

The original superport campaign checked 648 \(Z\)-minors, 568 response minors and 80 denominator identities, for 1,296 exact determinant/forest equalities. The revised adversarial campaign adds 2,197 exact identities across 17 configurations. It includes \(q=5\) all-minors replay, a complete \(q=4\) case, all nine root choices of a \(3+3\) partition, a network with genuine interior vertices, the single-superport reduction, multiple superports including singletons, the \(q=1\) endpoint, and a preprocessing robustness case containing parallel edges and an electrically inert loop. Every nonzero quotient-incidence determinant observed was \(\pm1\), and every identity agreed exactly.

\begin{table}[t]
\centering
\caption{Revised exact verification of the superport formulas.}
\begin{tabular}{lrrr}
\toprule
Campaign & \(Z\) & \(L\) & denom.\\
\midrule
Original three campaigns & 648 & 568 & 80\\
New adversarial campaign & 1090 & 1090 & 17\\
\midrule
Total & 1738 & 1658 & 97\\
\bottomrule
\end{tabular}
\label{tab:superverify}
\end{table}

The preserved adversarial executable contains 1,090 \(Z\)-minor checks, 1,090 response-minor checks and 17 denominator checks, giving \(1090+1090+17=2197\). This corrects an earlier table transcription that split the same total as 1,098 and 1,082; no theorem value or aggregate count changes. The original 1,296-identity campaign was reconstructed from its published protocol, the adversarial 2,197-identity campaign was rerun with the explicit \(\varepsilon\)-coefficient path physically absent, and the 687 virtual-dipole plus 520 subdivision/Kirchhoff checks were replayed in the same master audit. The combined ledger is therefore
\begin{equation}
1296+2197+687+520=\boxed{4700},
\end{equation}
with zero discrepancies. A second deterministic execution produced byte-identical result files. The SHA-256 digest was
\begin{center}
\small\texttt{96eca20abee9c6f4a17602e7b12343bc}\\[-1mm]
\small\texttt{730f012fe5f0fa9017a1d304fcb9f00f}.
\end{center}
The clean-room response coefficient was computed only as \(\widehat\chi_F(I)\widehat\chi_F(J)\); the retired explicit Jacobi-parity multiplier was not present in that implementation. An independent post-revision completion/covariance attack then generated 186 quotient-incidence configurations through \(q=8\) and performed 10,752 further exact checks. These separately tested the \(0,\pm1\) range, the nonzero-if-and-only-if spanning-tree criterion, the bridge identity \eqref{eq:completedbridge}, invariance of coefficient products under relabelling of non-ground quotient components, and covariance under reversal of coordinate orientations. All 10,752 checks passed.

\begin{table}[t]
\centering
\caption{New adversarial and extension verification.}
\begin{tabular}{lrr}
\toprule
Block & identities & failures\\
\midrule
Superport adversarial & 2197 & 0\\
Virtual-dipole extension & 687 & 0\\
Subdivision retained & 520 & 0\\
\bottomrule
\end{tabular}
\label{tab:newverify}
\end{table}

A representative \(q=5\) network checked all \(\sum_{r=0}^5\binom5r^2=252\) equal-cardinality \(Z\)-minors and all 252 response minors plus the denominator identity. A separate \(3+3\) six-vertex network was replayed under all \(3\times3=9\) root choices; each root choice checked 141 identities. These cases specifically target complement reversal and root-coordinate errors that smaller endpoint tests may miss.

\subsection{Prior-art comparison and originality}
Chaiken already gives signed spanning-forest formulas for arbitrary Laplacian minors \cite{chaiken}. Kenyon--Wilson give signed grove formulas for minors of ordinary electrical response matrices \cite{kenyonwilson}. Pylyavskyy--Shirokovskikh--Skopenkov prove the superport entry and determinant endpoints and explicitly pose the arbitrary-minor problem \cite{superport}. The present superport theorem combines these ingredients through the forest-dependent quotient incidence \(B_F=Q_FD\), yielding the completed-incidence sign rule \eqref{eq:alphahat}.

Lam--Lo--Yuen \cite{lamloyuen} prove an all-minors theorem for a discrete period matrix associated with a graph cellularly embedded on a closed oriented surface. Their coefficient for a homological quasi-tree is a product \(\det\mathcal T_I\det\mathcal T_J\) of homological intersection minors. The structural resemblance to \eqref{eq:alphahat} is real and should be cited. The determinant ranges, however, are different: in general their \(\det\mathcal T_I\) can have absolute value greater than one, while in the maximal \(2g\)-dimensional case they prove \(0,\pm1\). In the superport setting, every \(B_F\) is graphic incidence, so Lemma~\ref{lem:graphic} forces \emph{all} square minors to \(0,\pm1\), at every order. Their theory also depends on surface homology and quasi-trees; the present formula is embedding-free and its denominator is the valid-superport-forest polynomial.

The comparison is summarized schematically in Figure~\ref{fig:compare}. The point is not that determinant-product coefficients are new in isolation; it is that the superport boundary constraint collapses the all-minors sign to two completed graphic-incidence determinants with a universal coefficient bound.

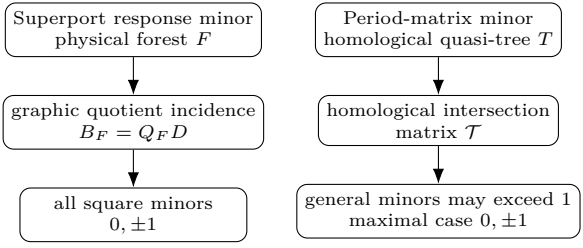
\begin{figure}[t]
\centering
\begin{tikzpicture}[every node/.style={font=\scriptsize,align=center}, box/.style={draw,rounded corners,inner sep=3pt,minimum width=30mm}]
\node[box] (s) {Superport response minor\\physical forest \(F\)};
\node[box,below=5mm of s] (sb) {graphic quotient incidence\\\(B_F=Q_FD\)};
\node[box,below=5mm of sb] (sc) {all square minors\\\(0,\pm1\)};
\node[box,right=7mm of s] (l) {Period-matrix minor\\homological quasi-tree \(T\)};
\node[box,below=5mm of l] (lb) {homological intersection\\matrix \(\mathcal T\)};
\node[box,below=5mm of lb] (lc) {general minors may exceed 1\\maximal case \(0,\pm1\)};
\draw[-{Latex[length=2mm]}] (s)--(sb); \draw[-{Latex[length=2mm]}] (sb)--(sc);
\draw[-{Latex[length=2mm]}] (l)--(lb); \draw[-{Latex[length=2mm]}] (lb)--(lc);
\end{tikzpicture}
\caption{Why the coefficient bounds differ from the related period-matrix formula of Lam--Lo--Yuen \cite{lamloyuen}.}
\label{fig:compare}
\end{figure}

The conductance-weighted subdivision strand has a separate boundary. Li--Yan treat weighted Laplacian matching polynomials \cite{liyan}; Que--Chen give the ordinary general-graph subdivision formula \cite{quechen}; Sun--Yang--Xu add vertex weights \cite{sunyangxu}. Theorem~\ref{thm:kirchhoffweighted} adds arbitrary physical edge conductances and the cycle factor \(c(C)\) to that resistance formula.

A broader boundary-condition preprint has been announced by the authors of \cite{superport} but was not public when this revision was frozen. The virtual-dipole theorem is therefore stated as a proof-generated extension, with priority deliberately left open pending direct comparison.

\section{Conclusion}
The central result is a direct all-minors matrix-tree theorem for the superport response matrix. The proof reorganizes the problem around the inverse-Gram representation \(L=(D^TK^{-1}D)^{-1}\). Each physical spanning forest produces a quotient port graph, and the reduced incidence of that quotient supplies the entire forest sign. Selector-row completion packages the complementary orientation into the determinant \(\widehat\chi_F\), so a forest contributes to \(\det L[I,J]\) exactly when the two complementary coordinate sets select spanning trees of the same quotient graph, with coefficient \(\widehat\chi_F(I)\widehat\chi_F(J)\). Jacobi parity remains part of one derivation route but disappears from the final theorem.

The strongest structural point of the revision is that the \(0,\pm1\) bound is no longer merely asserted from a named total-unimodularity fact or supported by tests. Equation~\eqref{eq:columnform} shows directly that every quotient column is \(0\), \(\pm e_a\), or \(e_a-e_b\), and the elementary induction in Lemma~\ref{lem:graphic} proves every square minor is \(0\) or \(\pm1\). This also cleanly separates the theorem from the related homological determinant products of Lam--Lo--Yuen, where the underlying matrices need not be graphic incidence matrices.

The principal-minor sector admits a further global consequence: its forest numerators assemble into a real-stable multiaffine polynomial. Hence the combinatorial numerators obey Hadamard--Fischer log-submodularity and, after normalization, define a strongly Rayleigh subset measure. The same determinant mechanism also extends to independent virtual dipole coordinates, showing that the essential hypothesis is graphic incidence rather than the specific rooted-star presentation.

The second strand shows that arbitrary positive edge conductances preserve the subdivision sign-reversing cancellation exactly; each surviving cycle simply acquires its conductance product. This yields conductance-weighted principal-minor, rooted-determinant and doubly weighted Kirchhoff identities.

The external claim is therefore specific. The paper supplies an explicit algebraic-combinatorial forest formula and sign rule for arbitrary minors of a superport response matrix, answering Problem~8.1 of \cite{superport} as stated, subject to independent proof and priority review. A 2026 period-matrix theorem supplies an important parallel determinant architecture but not the same superport formula. A forthcoming broader-boundary preprint should be compared before final priority language is frozen. The sigma-free clean-room ledger contains 4,700 exact identities with zero discrepancies, including a byte-identical second replay. These checks are supporting diagnostics rather than substitutes for proof.

\appendix
\section{Determinantal sign form of the all-minors forest identity}
Orient the physical edges arbitrarily and let \(\partial_g\) be the reduced vertex-edge incidence matrix after deleting the ground row. With \(W=\diag(c_e)\),
\begin{equation}
K=\partial_gW\partial_g^T.
\end{equation}
For a forest \(F\) with \(r+1\) components, deleting rows indexed by an \(r\)-set \(S\) from \(\partial_g[:,F]\) gives a square matrix precisely when each non-ground component is anchored once by \(S\). Its determinant vanishes otherwise and is \(\pm1\) in the anchored case. Complementary-minor duality identifies this sign with \(\det Q_F[:,S]\), with the usual positional parity already contained in Jacobi's identity. Thus
\begin{align}
&\eps_{V\setminus\{g\}}(S)\eps_{V\setminus\{g\}}(T)
 \det K[T^c,S^c]\notag\\
&\qquad=\sum_{F\in\cF_{r+1}}
 \det Q_F[:,S]\det Q_F[:,T]w(F).
\label{eq:appendixsign}
\end{align}
Multiplication by \(D\) after the forest sum and Cauchy--Binet then convert \(Q_F\) into \(Q_FD=B_F\). Lemma~\ref{lem:completed} subsequently absorbs the remaining complementary-minor parity into selector-row completion, which is why the final theorem can be stated purely as a product of completed quotient-incidence determinants rather than an auxiliary permutation or explicit Jacobi-sign sum.

\section{Weighted subdivision cancellation in full detail}
Fix \(X\subseteq V\) and expand
\begin{equation}
\det L_c[X,X]=\sum_{\pi\in S_X}\sgn(\pi)\prod_{v\in X}(L_c)_{v,\pi(v)}.
\end{equation}
For each fixed point, expand \((L_c)_{vv}=\sum_{e\ni v}c_e\). Call an edge \(e=uv\) bad in a completely expanded term when either \((uv)\) is a transposition or \(u,v\) are fixed and both choose \(e\). Toggle the least bad edge in a fixed ordering. The transposition term has factor \(-c_e^2A\), while the double fixed choice has \(+c_e^2A\); this is a fixed-point-free sign-reversing involution.

After cancellation, no transpositions remain and no two fixed vertices choose the same edge. Every nontrivial permutation cycle has length at least three and traces a physical cycle. Summing both orientations yields \(-2\) times the product of its conductances. For a vertex-disjoint cycle union \(C\), the factor is \((-2)^{p(C)}c(C)\). The remaining fixed vertices choose distinct incident edges outside \(E(C)\), in weight-preserving bijection with a subdivision matching saturating every remaining original vertex. This proves Theorem~\ref{thm:subdivision}.

\section{Exact verification protocol}
The exact verification used the following independent reconstruction path.
\begin{enumerate}[leftmargin=*,itemsep=1pt]
\item Generate a connected graph and assign rational conductances.
\item Construct \(L_c\), choose superports and roots, and form \(K,D,Z,L\) with exact rational linear algebra.
\item Enumerate spanning forests by edge subsets and reject cyclic subsets by union--find.
\item For every forest, construct its components, \(Q_F\), \(B_F=Q_FD\), and integer incidence determinants.
\item Compare every selected \(Z\)-minor with Theorem~\ref{thm:gram} and every selected response minor with Theorem~\ref{thm:allminors}.
\item Independently evaluate the valid-forest denominator.
\item For the virtual-dipole campaign, replace \(D\) by the incidence matrix of an oriented virtual forest and repeat all equal-cardinality minors.
\item Separately enumerate 2-regular subgraphs and saturated subdivision matchings for Theorems~\ref{thm:subdivision}--\ref{thm:kirchhoffweighted}.
\end{enumerate}
The audit additionally asserts \(|\chi_F|\le1\) for every square quotient-incidence minor encountered. The new adversarial implementation is supplied with the publication package so that the exact counts in Tables~\ref{tab:superverify}--\ref{tab:newverify} can be replayed.

\section{Notation table}
\begin{center}
\begin{tabular}{ll}
\toprule
Symbol & Meaning\\
\midrule
\(c_e\) & conductance of physical edge \(e\)\\
\(L_c\) & weighted graph Laplacian\\
\(\tau_c(G)\) & weighted spanning-tree polynomial\\
\(A_a\) & a superport\\
\(R,N\) & roots and nonroot coordinates\\
\(K\) & Laplacian grounded at \(g\)\\
\(D\) & port-difference incidence matrix\\
\(Z\) & \(D^TK^{-1}D\)\\
\(L\) & superport response \(Z^{-1}\)\\
\(Q_F\) & non-ground component indicator\\
\(H_F\) & quotient port multigraph\\
\(B_F\) & reduced incidence \(Q_FD\)\\
\(\chi_F(S)\) & oriented quotient-incidence minor\\
\(\widehat\chi_F(I)\) & completed quotient-incidence determinant\\
\(\Phi_{\valid}\) & valid-superport-forest sum\\
\(\Psi_I\) & principal forest numerator\\
\bottomrule
\end{tabular}
\end{center}

\paragraph{Acknowledgments.}
The author thanks Pavel Pylyavskyy, Svetlana Shirokovskikh, and Mikhail Skopenkov for the superport framework and open Problem~8.1 that motivated this investigation.

\paragraph{Computational assistance disclosure.}
Computational algorithms, proof-search procedures, and certificate criteria were developed by the author. Generative AI systems, including ChatGPT and Claude, served as computational assistants and were used for debugging and cross-checking. Mathematical claims were accepted strictly on the basis of independently reproducible exact, symbolic, rational, interval, exhaustive, or controlled high-precision certificates, and were independently verified.

\end{document}